\documentclass[11pt,letterpaper]{article}

\usepackage{fullpage}
\usepackage{amsmath,amsthm,amssymb}
\usepackage{enumerate}
\usepackage{color}

\newtheorem{theorem}{Theorem}
\newtheorem{proposition}[theorem]{Proposition}

\newtheorem{lemma}[theorem]{Lemma}
\newtheorem{conjecture}[theorem]{Conjecture}

\allowdisplaybreaks[3]

\long\def\symbolfootnote[#1]#2{\begingroup%
\def\thefootnote{\fnsymbol{footnote}}\footnote[#1]{#2}\endgroup} 

\newcommand{\C}{\mathbb{C}}
\newcommand{\R}{\mathbb{R}}
\newcommand{\Z}{\mathbb{Z}}
\newcommand{\anglebracket}[1]{\left\langle#1\right\rangle}

\newcommand{\Real}{{\rm Re}}

\newcommand{\rank}{{\rm rank}}
\newcommand{\ba}{{\bf a}}

\newcommand{\bz}{{\bf 0}}
\newcommand{\abs}[1]{\lvert#1\rvert}
\newcommand{\conj}[1]{\overline{#1}}

\newcommand{\Biggabs}[1]{\Bigg\lvert#1\Bigg\rvert}
\newcommand{\stack}[2]{\genfrac{}{}{0pt}{1}{#1}{#2}}

\begin{document}

\title{An infinite family of strongly unextendible \\mutually unbiased bases in~$\C^{2^{2h}}$}
\author{Jonathan Jedwab \and Lily Yen}
\date{16 April 2016}
\maketitle

\symbolfootnote[0]{
J.~Jedwab is with Department of Mathematics, 
Simon Fraser University, 8888 University Drive, Burnaby BC V5A 1S6, Canada.
\par
L.~Yen is with Department of Mathematics and Statistics,
Capilano University, 2055 Purcell Way, North Vancouver BC V7J 3H5, Canada
and Department of Mathematics, 
Simon Fraser University, 8888 University Drive, Burnaby BC V5A 1S6, Canada.
\par
J.~Jedwab is supported by NSERC.
\par
Email: {\tt jed@sfu.ca}, {\tt lyen@capilanou.ca}
\par
2010 Mathematics Subject Classification 05B20, 15A30, 81P45, 94B60
\par
}

\begin{abstract}
A set of $b$ mutually unbiased bases (MUBs) in $\C^d$ (for $d > 1$) comprises $bd$ vectors in $\C^d$, partitioned into $b$ orthogonal bases for $\C^d$ such that the pairwise angle between all vectors from distinct bases is $\arccos(1/\sqrt{d})$. The largest number $\mu(d)$ of MUBs that can exist in $\C^d$ is at most $d+1$, but constructions attaining this bound are known only when $d$ is a prime power. A set of $b$ MUBs in $\C^d$ that cannot be enlarged, even by the first vector of a potential $(b+1)$-th MUB, is called strongly unextendible.
Until now, only one infinite family of dimensions $d$ containing $b(d)$ strongly unextendible MUBs in $\C^d$ satisfying $b(d) < \mu(d)$ was known; this family, due to Sz{\'a}nt{\'o}, is asymptotically ``large'' in the sense that $b(d)/\mu(d) \to 1$ as $d \to \infty$. 
However, the existence of $2^{m-1}+1$ strongly unextendible MUBs in $\C^{2^m}$ for each integer $m > 1$ has been conjectured by Mandayam et al.
We prove their conjecture for all even values of $m$, using only elementary linear algebra. The existence of this ``small'' new infinite family suggests, contrary to widespread belief, that $\mu(d)$ for non-prime-powers $d$ might be significantly larger than the size of particular unextendible sets. 
\end{abstract}

\section{Introduction}
The Hermitian inner product of vectors
$A=(A(x))_{0 \le x < d}$ and $B=(B(x))_{0 \le x < d}$ in $\C^d$ is
$\anglebracket{A,B} = \sum_{x=0}^{d-1}A(x) \conj{B(x)}$. The \emph{angle} between $A$ and $B$ is 
$\arccos(\frac{|\anglebracket{A,B}|}{||A|| \cdot ||B||})$, where $||A|| = \sqrt{\anglebracket{A,A}}$ is the norm of~$A$.
A set of $b$ \emph{mutually unbiased bases} (MUBs) in $\C^d$ (for $d > 1$) comprises $bd$ vectors in $\C^d$, partitioned into $b$ orthogonal bases for $\C^d$ such that the pairwise angle between vectors in distinct bases is $\arccos(1/\sqrt{d})$. 
After applying a unitary transformation and normalizing, we may assume that one of the bases is $\sqrt{d}$ times the standard basis and therefore that each component of each vector in all other bases has unit magnitude.
The angle condition for these other bases is then $|\anglebracket{A,B}|=\sqrt{d}$ for vectors $A,B$ in distinct bases.
For example, the following sets of vectors form 5 MUBs in $\C^4$:
\[
\hspace{-6pt}
\begin{array}{rrrr}
(2 & 0 & 0 & 0) \\
(0 & 2 & 0 & 0) \\
(0 & 0 & 2 & 0) \\
(0 & 0 & 0 & 2) 
\end{array}
\quad
\begin{array}{rrrr}
(1 &  1 &  1 &  1) \\
(1 &  1 & -1 & -1) \\
(1 & -1 &  1 & -1) \\
(1 & -1 & -1 &  1) 
\end{array}
\quad
\begin{array}{rrrr}
(1 &  1 &  i & -i) \\
(1 &  1 & -i &  i) \\
(1 & -1 &  i &  i) \\
(1 & -1 & -i & -i) 
\end{array}
\quad
\begin{array}{rrrr}
(1 &  i &  1 & -i) \\
(1 &  i & -1 &  i) \\
(1 & -i &  1 &  i) \\
(1 & -i & -1 & -i) 
\end{array}
\quad
\begin{array}{rrrr}
(1 &  i &  i & -1) \\
(1 &  i & -i &  1) \\
(1 & -i &  i &  1) \\
(1 & -i & -i & -1) 
\end{array}
\]
Schwinger \cite{Schwinger} introduced MUBs in 1960, noting that when a quantum system is prepared in a state belonging to one basis, all outcomes of measurement with respect to any other basis are equally probable and therefore convey no information about the system.
The term ``mutually unbiased bases'' was introduced by Wootters and Fields in 1989 \cite{Wootters-Fields}.
The MUB property can be exploited in secure quantum key exchange \cite{Bennett-Brassard}, quantum state determination \cite{Ivanovic-geometrical}, quantum state reconstruction \cite{Wootters-Fields}, and detection of quantum entanglement \cite{Spengler-entanglement}; see \cite{Durt-survey} for a comprehensive survey of research on MUBs up to 2010.
There are intriguing connections between MUBs and various combinatorial structures, including finite projective planes \cite{Saniga-Planat-Rosu}, mutually orthogonal Latin squares \cite{Wocjan-Beth}, relative difference sets \cite{Godsil-Roy}, complex Hadamard matrices~\cite{Szollosi-thesis}, and complex equiangular lines \cite{Jedwab-Wiebe-MUBs-lines}. 

The central problem is to determine the largest number $\mu(d)$ of MUBs that can exist in $\C^d$. 
Following Grassl \cite{Grassl-icqft-slides}, we call a set of $b$ MUBs in $\C^d$ that cannot be enlarged to a set of size $b+1$ MUBs \emph{$\C$-unextendible}, and a set that cannot be enlarged by even one vector of a potential $(b+1)$-th MUB \emph{strongly $\C$-unextendible}; in the latter case, we say there is no vector in $\C$ that is unbiased with respect to each vector of the MUBs.
Corresponding definitions apply for MUBs in $\R^d$ and for (strongly) $\R$-unextendible sets.

More than forty years ago, Delsarte, Goethals, and Seidel \cite{DGS-bounds} used Jacobi polynomials to establish an upper bound on $\mu(d)$ and on the corresponding quantity for MUBs in~$\R^d$.
\begin{theorem}{\rm {\cite[Table~I with $\alpha =1/d$ and $\beta=0$]{DGS-bounds}}}
\label{thm:maxnoMUBs}
\mbox{}
\begin{enumerate}[(i)]
\item
The number $\mu(d)$ of MUBs that can exist in $\C^d$ is at most $d+1$.
Every set of $d+1$ MUBs in $\C^d$ is strongly $\C$-unextendible.
\label{CdMUBs}
\item
The number of MUBs that can exist in $\R^d$ is at most $d/2+1$.
Every set of $d/2+1$ MUBs in $\R^d$ is strongly $\R$-unextendible.
\end{enumerate}
\end{theorem}

The following lower bound on $\mu(d)$, arising from a product construction, is due to Klappenecker and R\"{o}tteler \cite{Klappenecker-Rotteler}. 
\begin{theorem}{\rm {\cite[Lemma~3]{Klappenecker-Rotteler}}}
\label{thm:product}
Let $d, d' > 1$. Then $\mu(dd') \ge \min(\mu(d),\mu(d'))$. 
\end{theorem}
The upper bound $d+1$ on $\mu(d)$ in Theorem~\ref{thm:maxnoMUBs}~\eqref{CdMUBs} is attained when $d$ is a prime power \cite{Ivanovic-geometrical}, \cite{Wootters-Fields}. It follows from Theorem~\ref{thm:product} that, for distinct primes $p_1, p_2, \dots, p_r$ and positive integers $a_1, a_2, \dots, a_r$, we have $\mu(p_1^{a_1}p_2^{a_2}\dots p_r^{a_r}) \ge 1+ \min_i p_i^{a_i}$; a stronger lower bound can be obtained for infinitely many dimensions using sets of mutually orthogonal Latin squares~\cite{Wocjan-Beth}.
However, it is not known whether the upper bound on $\mu(d)$ in Theorem~\ref{thm:maxnoMUBs}~\eqref{CdMUBs} is attained for even a single value of $d>1$ that is not a prime power. 
Indeed, the determination of $\mu(d)$ for non-prime-powers $d$ was proposed in 2006 as one of \emph{The ten most annoying questions in quantum computing} (by virtue of having ``caused all would-be climbers to fall flat on their asses''!) \cite{Aaronson-topten}; in 2014, only this question and two others from the original list remained unanswered \cite{Aaronson-toptennew}.
It is therefore interesting to pose the question: When and how can a set of MUBs be extended and, if it cannot, then when and why is it strongly unextendible? We now summarise the few known general results addressing this question.

\begin{theorem}{\rm \cite{Weiner}}
\label{thm:notd} 
Every set of $d$ MUBs in $\C^d$ is extendible to a set of $d+1$ MUBs in $\C^d$ (that is, $\mu(d) \ne d$).
\end{theorem}
In view of Theorems~\ref{thm:maxnoMUBs}, \ref{thm:product},~\ref{thm:notd}, the current state of knowledge for the smallest non-prime-power dimension 6 is that $\mu(6) \in \{3, 4, 5, 7\}$.
Several constructions of infinite families of sets of 3 MUBs in $\C^6$ are known
\cite[p.~57]{Zauner-thesis}, \cite[Appendix~B]{Jaming-Pauli}, \cite{Szollosi-two-parameter}, but no set of 4 MUBs in $\C^6$ has been found. Indeed, Zauner \cite[p.~57]{Zauner-thesis} conjectured in 1999 that no such set exists. 
In 2007, Bengtsson \cite{Bengtsson-three-ways} reported ``a growing consensus'' in favour of this conjecture, yet concluded that ``We have almost no evidence either way''. Three years later, Durt et al. \cite{Durt-survey} considered that ``the evidence for [Zauner's] conjecture is overwhelming, but not quite conclusive''.
Two pieces of supporting evidence for the conjecture are: a computational proof that if at least one of a set of 3 MUBs in $\C^6$ is constrained to belong to the ``Fourier family $F(a,b)$'' (a generalization of the Fourier matrix of order~6) then the set is $\C$-unextendible~\cite{Jaming-Pauli}; and a proof that every set of 3 MUBs in $\C^6$ arising from the product construction leading to Theorem~\ref{thm:product} is strongly $\C$-unextendible~\cite{McNulty-Weigert}.

Until now, only one infinite family of dimensions $d$ containing $b(d)$ strongly unextendible MUBs in $\C^d$ satisfying $b(d) < \mu(d)$ was known, due to Sz{\'a}nt{\'o}.
\begin{theorem}{\rm \cite{Szanto}}
\label{thm:szanto}
For each prime $p$ congruent to $3$ modulo~$4$, there exists a set of $p^2-p+2$ strongly $\C$-unextendible MUBs in~$\C^{p^2}$.
For $p = 2, 3, 5, 7, 11$ there also exists a set of $p^2-1$ strongly $\C$-unextendible MUBs in~$\C^{p^2}$.
\end{theorem}

The motivation for this paper is provided by two sets of strongly $\C$-unextendible MUBs and an accompanying conjecture recently presented by Mandayam et al. \cite{Mandayam-Pauli}.
\begin{theorem}{\rm {\cite[Section~4]{Mandayam-Pauli}}}
\label{thm:35strong}
There exist $3$ strongly $\C$-unextendible MUBs in $\C^4$, and $5$ strongly $\C$-unextendible MUBs in $\C^8$.
\end{theorem}
\begin{conjecture}{\rm {\cite[Conjecture~1]{Mandayam-Pauli}}}
\label{conj:family}
For each integer $m > 1$, there exists a set of $2^{m-1}+1$ strongly $\C$-unextendible MUBs in $\C^{2^m}$. 
\end{conjecture}
Strong $\C$-unextendibility in Theorem~\ref{thm:35strong} was verified computationally in \cite{Mandayam-Pauli}, using Gr\"{o}bner basis techniques.
The sets of MUBs in Theorem~\ref{thm:35strong} were constructed from maximal commuting classes of Pauli operators, and Conjecture~\ref{conj:family} was stated in \cite{Mandayam-Pauli} as holding specifically for such sets.
K.~Thas \cite{Thas} subsequently showed that Conjecture~\ref{conj:family} is false for all $m > 3$ when this restriction is applied, but proposed that the conjecture holds for all $m>1$ using MUBs constructed from complete partial spreads \cite[Conjecture~8.6]{Thas}.

The main result of ths paper is Theorem~\ref{thm:fullrankfamily}, which establishes Conjecture~\ref{conj:family} (without reference to complete partial spreads) for all even~$m$.
\begin{theorem}
\label{thm:fullrankfamily}
\mbox{}
For each integer $h \ge 1$, there exists a set of $2^{2h-1}+1$ strongly $\C$-unextendible MUBs in $\C^{2^{2h}}$. 
\end{theorem}
Our proof of Theorem~\ref{thm:fullrankfamily} uses only elementary linear algebra, and does not rely at all on computation. We specify the sets of MUBs described in Theorem~\ref{thm:fullrankfamily} explicitly; in fact, they are MUBs in $\{1,-1\}^{2^{2h}}$ that have long been known to attain the upper bound of Theorem~\ref{thm:maxnoMUBs}~(ii) when $d$ is a power of~$4$ (see Proposition~\ref{prop:benttoMUBs} and Theorem~\ref{thm:kerdock}). The new and surprising result is that these MUBs, which are strongly $\R$-unextendible by Theorem~\ref{thm:maxnoMUBs}~(ii), are also strongly $\C$-unextendible.
Theorem~\ref{thm:fullrankfamily} gives the first known infinite family of $b(d)$ strongly $\C$-unextendible MUBs in $\C^d$ for which $\lim_{d \to \infty} b(d)/\mu(d) < 1$. The existence of this family suggests that caution is warranted, for example, in interpreting the existence of sets of 3 $\C$-unextendible MUBs in $\C^6$ \cite{Jaming-Pauli}, \cite{McNulty-Weigert} as evidence that $\mu(6) < 7$, especially when the sets are constrained to satisfy some structural condition; indeed, we see that for $d=2^{2h}$ there exist sets of $d/2+1$ strongly $\C$-unextendible MUBs in $\C^d$ (constrained actually to lie in $\R^d$) even though $\mu(d) = d+1$.

In Section~\ref{sec:prelims}, we shall provide required background on Boolean functions and bent functions, including short proofs of some known results with the intention of making the paper more accessible. 
In Section~\ref{sec:proof}, we shall prove Theorem~\ref{thm:fullrankfamily}.

\section{Boolean functions and bent functions}
\label{sec:prelims}

A \emph{Boolean function on $\Z_2^m$} is a function $g:\Z_2^m \to \Z_2$.
The corresponding vector $(g(x))_{x \in \Z_2^m} \in \Z_2^{2^m}$ is the evaluation of $g(x)$ at the $2^m$ points of $\Z_2^m$ taken in lexicographic order.
For example, the vector corresponding to the Boolean function 
$g(x_1,x_2,x_3,x_4) = x_1x_2 + x_1x_3 + x_2x_4$ on $\Z_2^4$ is
$(0 0 0 0 0 1 0 1 0 0 1 1 1 0 0 1)$, whose initial element is $g(0,0,0,0)$ and whose final element is $g(1,1,1,1)$. See \cite{Carlet} and \cite{Carlet-Mesnager}, for example, for detailed background on Boolean functions.

The \emph{Walsh-Hadamard transform} of a Boolean function $g$ on $\Z_2^m$ is the function $\widehat{g}:\Z_2^m\to\Z$ given by
\[
\widehat{g}(u) = \sum_{x\in\Z_2^m} (-1)^{g(x)+u\cdot x} \quad \mbox{for $u \in \Z_2^m$},
\]
where $\cdot$ is the usual inner product in $\Z_2^m$.
A Boolean function $g$ on $\Z_2^m$ is \emph{bent} if
\[
\widehat{g}(u) \in \{2^{m/2}, -2^{m/2}\} \quad \mbox{for all $u \in \Z_2^m$}.
\]
Bent functions exist for all positive even integers~$m$.

A \emph{bent set on $\Z_2^{2h}$} is a finite set of Boolean functions on $\Z_2^{2h}$ for which the sum of any two distinct functions in the set is bent. 
We may assume (by adding one function to all the others) that one element of the set is the zero function, and then all the other elements are themselves bent.
For example, a bent set of size 8 on $\Z_2^4$ is given by the Boolean functions 
\begin{align*}
& 0, \quad
x_1x_2 + x_3x_4, \quad
x_1x_2 + x_1x_3 + x_1x_4 + x_2x_3, \quad 
x_1x_2 + x_1x_3 + x_2x_4, \quad
x_1x_2 + x_1x_4 + x_2x_3 + x_2x_4, \\
&x_1x_3 + x_1x_4 + x_2x_4 + x_3x_4, \quad
x_1x_3 + x_2x_3 + x_2x_4 + x_3x_4, \quad
x_1x_4 + x_2x_3 + x_3x_4.
\end{align*}
A bent set on $\Z_2^{2h}$ can be used to construct a set of real MUBs in $\{1,-1\}^{2^{2h}}$, as we now describe.
Write $I[\cdot]$ for the indicator function.

\begin{proposition}{\rm \cite{Cameron-Seidel} (see also \cite{Calderbank-Z4-Kerdock}, \cite{Kantor-codes})}
\label{prop:benttoMUBs}
Suppose $\{g_1, g_2, \dots, g_r\}$ is a bent set on~$\Z_2^{2h}$. 
Then $2^h$ times the standard basis for $\R^{2^{2h}}$, together with the $r$ sets of $2^{2h}$ vectors
$\big\{\big((-1)^{g_j(x)+ u \cdot x}\big)_{x \in \Z_2^{2h}} : u \in \Z_2^{2h}\big\}$ for $1 \le j \le r$, form $r+1$ MUBs in $\{1,-1\}^{2^{2h}}$.
\end{proposition}

\begin{proof}
For each $j$, the $2^{2h}$ vectors of $B_j = \big\{\big((-1)^{g_j(x)+ u \cdot x}\big)_{x \in \Z_2^{2h}} : u \in \Z_2^{2h}\big\}$ form an orthogonal basis for~$\R^{2^{2h}}$, because for distinct $u, v \in \Z_2^{2h}$ we have
\[
\anglebracket{\big((-1)^{g_j(x)+u\cdot x}\big)_x, \big((-1)^{g_j(x)+v\cdot x}\big)_x} =
\sum_{x\in\Z_2^{2h}} (-1)^{(u+v)\cdot x} = 0
\]
using the identity
\begin{equation}
\sum_{w \in \Z_2^m} (-1)^{w\cdot z} = 2^m I[z = 0] \quad \mbox{for all $z \in \Z_2^m$}.
\label{indicator}
\end{equation}

The vectors from distinct bases $B_j$ and $B_k$ are mutually unbiased, because for $u, v \in \Z_2^{2h}$ we have
\[
\anglebracket{\big((-1)^{g_j(x)+u\cdot x}\big)_x, \big((-1)^{g_k(x)+v\cdot x}\big)_x} =
\sum_{x\in\Z_2^{2h}} (-1)^{(g_j+g_k)(x)+(u+v)\cdot x} =
\widehat{g_j+g_k}(u+v),
\]
which has magnitude $\sqrt{2^{2h}}$ because $g_j+g_k$ is a bent function on $\Z_2^{2h}$ for distinct $j,k$.
\end{proof}

The following existence result for bent sets is due to Kerdock~\cite{Kerdock}.

\begin{theorem}{\rm \cite{Kerdock}, \cite[p.~456]{MacWilliams-Sloane}} 
\label{thm:kerdock}
For each integer $h \ge 1$, there exists a bent set of size $2^{2h-1}$ on $\Z_2^{2h}$.
\end{theorem}
Application of Proposition~\ref{prop:benttoMUBs} to the bent set of Theorem~\ref{thm:kerdock} produces a set of $2^{2h-1}+1$ MUBS in $\{1,-1\}^{2^{2h}}$, which attains the upper bound in Theorem~\ref{thm:maxnoMUBs}~(ii) for the number of MUBs in $\R^d$ when $d = 2^{2h}$. 

We require two further auxiliary results. Write $(\Z_2^m)^*$ for $\Z_2^m \setminus \{0\}$.

\begin{proposition}{\rm {\cite[p.79]{Carlet}}}
Suppose $g(x)$ is a bent function on $\Z_2^{2h}$, and let $a \in (\Z_2^{2h})^*$. Then
\[
\sum_{x \in \Z_2^{2h}} (-1)^{g(x)+g(x+a)} = 0. 
\]
\label{prop:balanced-derivative}
\end{proposition}

\begin{proof}
Since $g(x)$ is bent, we have
\[
2^{2h} 
 = |\hat{g}(u)|^2 
 = \sum_{x,y \in \Z_2^{2h}}(-1)^{g(x)+u\cdot x} (-1)^{g(y)+u\cdot y} 
 = \sum_{x,b \in \Z_2^{2h}}(-1)^{g(x)+g(x+b)} (-1)^{u\cdot b}
\]
by setting $y = x+b$. Multiply the first and last expressions by $(-1)^{u \cdot a}$ and sum over $u \in \Z_2^{2h}$ to give
\[
2^{2h} \sum_{u \in \Z_2^{2h}} (-1)^{u\cdot a} 
 = \sum_{x,b \in \Z_2^{2h}}(-1)^{g(x)+g(x+b)} \sum_{u \in \Z_2^{2h}} (-1)^{u\cdot (a+b)}.
\]
The result follows by applying \eqref{indicator} to the sum over $u$ on both sides.
\end{proof}

\begin{lemma}
The $2^m$ vectors $\big\{\big((-1)^{u\cdot \ell}\big)_{u \in \Z_2^m} : \ell \in \Z_2^m\big\}$ are pairwise orthogonal, and therefore linearly independent over~$\R$.
\label{lem:lin-ind}
\end{lemma}

\begin{proof}
For distinct $k, \ell \in \Z_2^m$, we have
$
\anglebracket{\big((-1)^{u \cdot k}\big)_u, \big((-1)^{u \cdot \ell}\big)_u} =
\sum_{u \in \Z_2^m} (-1)^{u \cdot (k+\ell)} = 0
$
by~\eqref{indicator}.
\end{proof}

\section{Proof of Theorem~\ref{thm:fullrankfamily}}
\label{sec:proof}
\begin{proof}[Proof of Theorem~\ref{thm:fullrankfamily}]
 From Theorem~\ref{thm:kerdock}, there exists a bent set $\{g_1, g_2, \dots, g_{2^{2h-1}}\}$ on~$\Z_2^{2h}$ and we may assume $g_1 = 0$.
 From Proposition~\ref{prop:benttoMUBs}, this bent set gives a set of $2^{2h-1}+1$ MUBs in $\{1,-1\}^{2^{2h}}$, comprising $2^h$ times the standard basis together with the $2^{2h-1}$ bases
$\big\{\big((-1)^{g_j(x)+u\cdot x}\big)_{x \in \Z_2^{2h}} : u \in \Z_2^{2h}\big\}$ for $1 \le j \le 2^{2h-1}$. We shall show that these MUBs are strongly $\C$-unextendible.

Suppose, for a contradiction, that the vector $\big(A(x)\big)_{x \in \Z_2^{2h}} \in \C^{2^{2h}}$ is unbiased with respect to each vector of these MUBs.
By reference to $2^h$ times the standard basis, each $A(x)$ has magnitude~$1$.
By reference to the other $2^{2h-1}$ bases, for $1 \le j \le 2^{2h-1}$ and $u \in \Z_2^{2h}$ we have
\[
\Biggabs{\sum_{x \in \Z_2^{2h}} A(x) (-1)^{g_j(x)+u\cdot x}} = 2^h
\]
and squaring yields 
\[
\sum_{x, y \in \Z_2^{2h}} A(x)\conj{A(y)} (-1)^{g_j(x)+g_j(y)+u\cdot (x+y)} = 2^{2h}.
\]
The terms of this sum for which $x=y$ contribute
$\sum_{x \in \Z_2^{2h}}\abs{A(x)}^2 = \sum_{x \in \Z_2^{2h}}1 = 2^{2h}$, and therefore
\begin{equation}
\sum_{\stack{x, y \in \Z_2^{2h}}{x \ne y}} A(x)\conj{A(y)} (-1)^{g_j(x)+g_j(y)+u\cdot (x+y)} = 0
\quad \mbox{for $1 \le j \le 2^{2h-1}$ and $u \in \Z_2^{2h}$}.
\label{xnoty}
\end{equation}
Order the elements of $\Z_2^{2h}$ lexicographically, writing $x<y$ 
to mean that $x$ precedes $y$ in this ordering. Define
\[
a_{x,y} = \tfrac{1}{2} \left(A(x)\conj{A(y)}+A(y)\conj{A(x)}\right) = \Real\left(A(x)\conj{A(y)}\right)
\quad \mbox{for $x, y \in \Z_2^{2h}$}
\]
and
\[
m_{j,u,x,y} = (-1)^{g_j(x)+g_j(y)+u\cdot(x+y)} \quad 
\]
Then from \eqref{xnoty} we have
\[
\sum_{\stack{x,y \in \Z_2^{2h}}{x<y}} m_{j,u,x,y} a_{x,y} = 0 \quad \mbox{for $1 \le j \le 2^{2h-1}$ and $u \in \Z_2^{2h}$},
\]
which is a homogeneous linear system of $2^{2h-1}\cdot 2^{2h}=2^{4h-1}$ equations in the $\binom{2^{2h}}{2} = 2^{2h-1}(2^{2h}-1)$ real variables $(a_{x,y})_{x < y}$. We can represent this system in the form $M \ba = \bz$ where $M = (m_{j,u,x,y})$ is the $2^{4h-1} \times 2^{2h-1}(2^{2h}-1)$ real matrix whose rows are indexed by $(j,u)$ and whose columns are indexed by $(x,y)$ with $x < y$, and $\ba = (a_{x,y})_{x < y}$ is a vector of $2^{2h-1}(2^{2h}-1)$ real entries.

Partition the columns of $M$ into $2^{2h}-1$ submatrices $M_\ell$ of size $2^{4h-1} \times 2^{2h-1}$, where $M_\ell$ is given by
\[
M_\ell = (m_{j,u,x,\ell+x}) = \big((-1)^{g_j(x) + g_j(\ell+x) + u\cdot \ell}\big)
\quad \mbox{for $\ell \in (\Z_2^{2h})^*$}.
\]
The rows of $M_\ell$ are indexed by $(j,u)$,
and the columns are indexed by $(x, \ell+x)$ for the $2^{2h-1}$ values of $x \in \Z_2^{2h}$ satisfying $x < \ell+x$.

For each $\ell \in (\Z_2^{2h})^*$, the first $2^{2h}$ entries of each column of $M_\ell$ are given by the vector $(m_{1,u,x,\ell+x})_{u \in \Z_2^{2h}} = \big((-1)^{u\cdot \ell}\big)_{u \in \Z_2^{2h}}$ (independently of $x$), using $g_1=0$. 
The set of all $2^{2h}-1$ such vectors, as $\ell$ ranges over $(\Z_2^{2h})^*$, is linearly independent over $\R$ by Lemma~\ref{lem:lin-ind}, and therefore
\[
\rank(M) = \sum_{\ell \in (\Z_2^{2h})^*} \rank(M_\ell).
\]
We claim that
\[
\rank(M_\ell) = 2^{2h-1} \quad \mbox{for each $\ell \in (\Z_2^{2h})^*$}.
\]
It then follows that $\rank(M) = 2^{2h-1}(2^{2h}-1)$, so $M$ has full rank.  The homogeneous linear system $M \ba = \bz$ therefore has only the trivial solution
\[
a_{x,y} = 0 \quad \mbox{for all $x < y$}.
\]
Writing $A(x) = e^{i \theta(x)}$ (using that each $A(x)$ has magnitude~1), this implies by the definition of $a_{x,y}$ that $\cos(\theta(x)-\theta(y)) = 0$ for all $x < y$. This is possible only if the vector $(A(x))$ contains at most 2 entries, which contradicts that the vector $(A(x))$ contains $2^{2h} \ge 4$ entries.

To prove the claim we note that, for $\ell \in (\Z_2^{2h})^*$, the $2^{2h-1}$ rows of $M_\ell$ given by
\[
(m_{j,0,x,\ell+x})_{x < \ell+x} = \left((-1)^{g_j(x) + g_j(\ell+x)}\right)_{x < \ell+x}
\quad \mbox{for $1 \le j \le 2^{2h-1}$}
\]
are pairwise orthogonal and therefore linearly independent over~$\R$: for distinct $j, k$ we have
\begin{align*}
\sum_{x < \ell+x} m_{j,0,x,\ell+x}m_{k,0,x,\ell+x}
 &= \sum_{\stack{x \in \Z_2^{2h}}{x < \ell+x}} (-1)^{g_j(x)+g_k(x) + g_j(\ell+x)+g_k(\ell+x)} \\
 &= \frac{1}{2}\sum_{x \in \Z_2^{2h}} (-1)^{(g_j+g_k)(x) + (g_j+g_k)(\ell+x)} \\
 &= 0
\end{align*}
by Proposition~\ref{prop:balanced-derivative}, because $g_j + g_k$ is bent and $\ell \in (\Z_2^{2h})^*$.
\end{proof}


\begin{thebibliography}{MBGW14}

\bibitem[Aar06]{Aaronson-topten}
S.~Aaronson.
\newblock The ten most annoying questions in quantum computing, August 2006.
\newblock http://www.scottaaronson.com/blog/?p=112.

\bibitem[Aar14]{Aaronson-toptennew}
S.~Aaronson.
\newblock The {NEW} ten most annoying questions in quantum computing, May 2014.
\newblock http://www.scottaaronson.com/blog/?p=1792.

\bibitem[BB14]{Bennett-Brassard}
C.~H. Bennett and G.~Brassard.
\newblock Quantum cryptography: {P}ublic key distribution and coin tossing.
\newblock {\em Theoret. Comput. Sci.}, 560, Part 1:7--11, 2014.

\bibitem[Ben07]{Bengtsson-three-ways}
I.~Bengtsson.
\newblock Three ways to look at mutually unbiased bases.
\newblock In {\em Foundations of Probability and Physics --- 4}, volume 889 of
  {\em AIP Conf. Proc.}, pages 40--51. Amer. Inst. Phys., New York, 2007.

\bibitem[Car10]{Carlet}
C.~Carlet.
\newblock Boolean functions for cryptography and error correcting codes.
\newblock In Y.~Crama and P.L. Hammer, editors, {\em Boolean Models and Methods
  in Mathematics, Computer Science, and Engineering}, volume 134 of {\em
  Encyclopedia Math. and its Applications}, pages 257--397. Cambridge Univ.
  Press, Cambridge, UK, 2010.

\bibitem[CCKS97]{Calderbank-Z4-Kerdock}
A.~R. Calderbank, P.~J. Cameron, W.~M. Kantor, and J.~J. Seidel.
\newblock {$\Z_4$}-{K}erdock codes, orthogonal spreads, and extremal
  {E}uclidean line-sets.
\newblock {\em Proc. London Math. Soc. (3)}, 75(2):436--480, 1997.

\bibitem[CM16]{Carlet-Mesnager}
C.~Carlet and S.~Mesnager.
\newblock Four decades of research on bent functions.
\newblock {\em Des. Codes Cryptogr.}, 78(1):5--50, 2016.

\bibitem[CS73]{Cameron-Seidel}
P.~J. Cameron and J.~J. Seidel.
\newblock Quadratic forms over {$GF(2)$}.
\newblock {\em Nederl. Akad. Wetensch. Proc. Ser. A vol. $76$ = Indag. Math.},
  35:1--8, 1973.

\bibitem[DEBZ10]{Durt-survey}
T.~Durt, B.-G. Englert, I.~Bengtsson, and K.~\.{Z}yczkowski.
\newblock On mutually unbiased bases.
\newblock {\em Int. J. Quantum. Inf.}, 8:535--640, 2010.

\bibitem[DGS75]{DGS-bounds}
P.~Delsarte, J.~M. Goethals, and J.~J. Seidel.
\newblock Bounds for systems of lines, and {J}acobi polynomials.
\newblock {\em Philips Res. Repts}, 30:91--105, 1975.

\bibitem[GR09]{Godsil-Roy}
C.~Godsil and A.~Roy.
\newblock Equiangular lines, mutually unbiased bases, and spin models.
\newblock {\em European J. Combin.}, 30(1):246--262, 2009.

\bibitem[Gra09]{Grassl-icqft-slides}
M.~Grassl.
\newblock Unextendible mutually unbiased bases, July 2009.
\newblock Slide presentation at International Conference on Quantum Foundations
  and Technology: Frontier and Future, Shanghai,
  http://quantum.ustc.edu.cn/old/conference/program2009/ppt file/Markus
  Grassl\_Grassl\_UnextendibleMUBs.pdf.

\bibitem[Iva81]{Ivanovic-geometrical}
I.~D. Ivanovi{\'c}.
\newblock Geometrical description of quantal state determination.
\newblock {\em J. Phys. A}, 14(12):3241--3245, 1981.

\bibitem[JMM{\etalchar{+}}09]{Jaming-Pauli}
P.~Jaming, M.~Matolcsi, P.~M{\'o}ra, F.~Sz{\"o}ll{\H{o}}si, and M.~Weiner.
\newblock A generalized {P}auli problem and an infinite family of
  {MUB}-triplets in dimension 6.
\newblock {\em J. Phys. A}, 42(24):245305, 25, 2009.

\bibitem[JW]{Jedwab-Wiebe-MUBs-lines}
J.~Jedwab and A.~Wiebe.
\newblock Constructions of complex equiangular lines from mutually unbiased
  bases.
\newblock {\em Des. Codes Cryptogr.}
\newblock Accepted, 2015. arXiv:1408.5169.

\bibitem[Kan95]{Kantor-codes}
W.M. Kantor.
\newblock Codes, quadratic forms and finite geometries.
\newblock In {\em Different Aspects of Coding Theory}, volume~50 of {\em Proc.
  Symp. Appl. Math.}, pages 153--177. Amer. Math. Soc., 1995.

\bibitem[Ker72]{Kerdock}
A.~M. Kerdock.
\newblock A class of low-rate nonlinear binary codes.
\newblock {\em Information and Control}, 20:182--187; ibid. 21 (1972), 395,
  1972.

\bibitem[KR04]{Klappenecker-Rotteler}
A.~Klappenecker and M.~R{\"o}tteler.
\newblock Constructions of mutually unbiased bases.
\newblock In {\em Finite fields and applications}, volume 2948 of {\em Lecture
  Notes in Comput. Sci.}, pages 137--144. Springer, Berlin, 2004.

\bibitem[MBGW14]{Mandayam-Pauli}
P.~Mandayam, S.~Bandyopadhyay, M.~Grassl, and W.~Wootters.
\newblock Unextendible mutually unbiased bases from {P}auli classes.
\newblock {\em Quantum Inf. Comput.}, 14(9\&10):823--844, 2014.

\bibitem[MS86]{MacWilliams-Sloane}
F.J. MacWilliams and N.J.A. Sloane.
\newblock {\em The {T}heory of {E}rror-{C}orrecting {C}odes}.
\newblock North-Holland, Amsterdam, 1986.

\bibitem[MW12]{McNulty-Weigert}
D.~McNulty and S.~Weigert.
\newblock On the impossibility to extend triples of mutually unbiased product
  bases in dimension six.
\newblock {\em Int. J. Quantum Inf.}, 10(5):1250056, 11, 2012.

\bibitem[Sch60]{Schwinger}
J.~Schwinger.
\newblock Unitary operator bases.
\newblock {\em Proc. Nat. Acad. Sci. U.S.A.}, 46:570--579, 1960.

\bibitem[SHB{\etalchar{+}}12]{Spengler-entanglement}
C.~Spengler, M.~Huber, S.~Brierley, T.~Adaktylos, and B.~C. Hiesmayr.
\newblock Entanglement detection via mutually unbiased bases.
\newblock {\em Phys. Rev. A}, 86:022311, Aug 2012.

\bibitem[SPR04]{Saniga-Planat-Rosu}
M.~Saniga, M.~Planat, and H.~Rosu.
\newblock Mutually unbiased bases and finite projective planes.
\newblock {\em J. Opt. B Quantum Semiclass. Opt.}, 6(9):L19--L20, 2004.

\bibitem[Sz{\'a}16]{Szanto}
A.~Sz{\'a}nt{\'o}.
\newblock Complementary decompositions and unextendible mutually unbiased
  bases.
\newblock {\em Linear Algebra Appl.}, 496:392--406, 2016.

\bibitem[Sz{\"o}10]{Szollosi-two-parameter}
F.~Sz{\"o}ll{\H{o}}si.
\newblock A two-parameter family of complex {H}adamard matrices of order 6
  induced by hypocycloids.
\newblock {\em Proc. Amer. Math. Soc.}, 138(3):921--928, 2010.

\bibitem[Sz{\"o}11]{Szollosi-thesis}
F.~Sz{\"o}ll{\H{o}}si.
\newblock {\em Construction, Classification and Parametrization of Complex
  Hadamard Matrices}.
\newblock PhD thesis, Central European University, 2011.

\bibitem[Tha]{Thas}
K.~Thas.
\newblock Unextendible mutually unbiased bases (after {M}andayam,
  {B}andyopadhyay, {G}rassl and {W}ootters).
\newblock arXiv:1407.2778 [quant-ph].

\bibitem[WB05]{Wocjan-Beth}
P.~Wocjan and T.~Beth.
\newblock New construction of mutually unbiased bases in square dimensions.
\newblock {\em Quantum Inf. Comput.}, 5(2):93--101, 2005.

\bibitem[Wei13]{Weiner}
M.~Weiner.
\newblock A gap for the maximum number of mutually unbiased bases.
\newblock {\em Proc. Amer. Math. Soc.}, 141:1963--1969, 2013.

\bibitem[WF89]{Wootters-Fields}
W.~K. Wootters and B.~D. Fields.
\newblock Optimal state-determination by mutually unbiased measurements.
\newblock {\em Ann. Physics}, 191(2):363--381, 1989.

\bibitem[Zau99]{Zauner-thesis}
G.~Zauner.
\newblock {\em Quantendesigns: Grundz\"uge einer nichtkommutativen
  Designtheorie}.
\newblock PhD thesis, University of Vienna, 1999.

\end{thebibliography}

\newcommand{\etalchar}[1]{$^{#1}$}

\end{document}